\documentclass[preprint,10pt]{elsarticle}

\usepackage{graphicx}
\usepackage{amssymb}
\usepackage{lineno}
\usepackage{amsthm}
\usepackage{marvosym}
\usepackage{amsmath}
\usepackage{color}
\usepackage{csquotes}

\usepackage{algorithmic}
\usepackage{algorithm}

\newtheorem{definition}{Definition}[section]

\newtheorem{theorem}{Theorem}[section]

\makeatletter
\def\ps@pprintTitle{%
   \let\@oddhead\@empty
   \let\@evenhead\@empty
   \let\@oddfoot\@empty
   \let\@evenfoot\@oddfoot
}
\makeatother
\begin{document}
	\begin{frontmatter}	
			\title{ A Theory of Rectangularly Dualizable Graphs }

	\author{Vinod Kumar\corref{cor1}\fnref{label2}}
	\ead{vinodchahar04@gmail.com}
	\author{Krishnendra Shekhawat\fnref{label2}}
	\ead{krishnendra.shekhawat@pilani.bits-pilani.ac.in}
	\fntext[label2]{Department of Mathematics, BITS Pilani, Pilani Campus, Rajasthan-333031, India }
	\cortext[cor1]{Corresponding Author}

\begin{abstract} A plane graph is called a {\it rectangular  graph} if  each of its edges  can be oriented  either horizontally or vertically, each of its interior regions is  a four-sided region and all interior regions  can be fitted in a rectangular enclosure. Only planar graphs can be dualized. If the dual of a plane  graph   is a   {\it rectangular  graph}, then  the plane graph   is a  {\it rectangularly dualizable graph}.  

In 1985, Ko{\'z}mi{\'n}ski and Kinnen 	presented a necessary and sufficient condition for the existence of a {\it rectangularly dualizable graph} for a separable connected plane graph. 
In this paper, we present a counter example for which the conditions given by them for   separable connected plane graphs fail and hence,
we derive a  necessary and sufficient condition for a  plane graph to be a {\it rectangularly dualizable graph}.
\end{abstract}
		
		
\end{frontmatter}

\section{Introduction}
\label{sec1}

  The theory of rectangularly dualizable graphs plays an important role in floorplanning, particularly at large scale such as VLSI circuit design. It provides  us information  at early stage to decide whether a given plane graph can be realized by a rectangular floorplan (RFP). There exists a  geometric duality relationship between plane graphs and rectangular floorplans (RFPs) which can be  described  as follows: 

  An RFP is a  partition  of a  rectangle $\mathcal{R}$ into $n$ rectangles $R_1, R_2, \dots, R_n$ such that no four of them meet at a point. A graph $\mathcal{G}_2$ is a dual of a plane graph $\mathcal{G}_1$  if the vertices of $\mathcal{G}_1$ correspond to the regions of $\mathcal{G}_1$ and for every pair of adjacent vertices of $\mathcal{G}_1$, the corresponding regions in $\mathcal{G}_2$ are adjacent.  A plane graph is called a {\it rectangular  graph} if  each of its edges  can be oriented  either horizontally or vertically, each of its interior regions is  a four-sided region  and all interior regions  can be fitted in a rectangular enclosure. Only planar graphs can be dualized. If dual of a plane  graph   is a   {\it rectangular  graph}, then  the plane graph   is a  {\it rectangularly dualizable graphs} (RDG). Thus an RFP can be seen as an embedding of the dual of a  planar graph and it can formally be described as a {\it rectangular dual graph} of an RDG, i.e.,  for  the dual of a RDG to be an RFP, we need to assign  horizontal  and vertical orientations to its edges. For a better clarification, consider a planar graph $\mathcal{G}$ shown in Fig.  \ref{fig:f1}a. We form its extended graph (Fig.  \ref{fig:f1}b.) by inserting cycle of  length 4 at the exterior  of $\mathcal{G}$ and then connecting the vertices of the cycle to the exterior vertices of $\mathcal{G}$. Then it is dualized in Fig.  \ref{fig:f1}c.   After assigning horizontal or vertical orientation to each of its edges, an embedding as shown in  Fig.  \ref{fig:f1}d.  is obtained. In fact, it is an RFP. Thus $\mathcal{G}$ is rectangularly dualized to an RFP.  This transformation is known as the rectangular dualization method which is well-studied in the literature.
 \begin{figure}[H]
 	\centering
 	\includegraphics[width=0.98\linewidth]{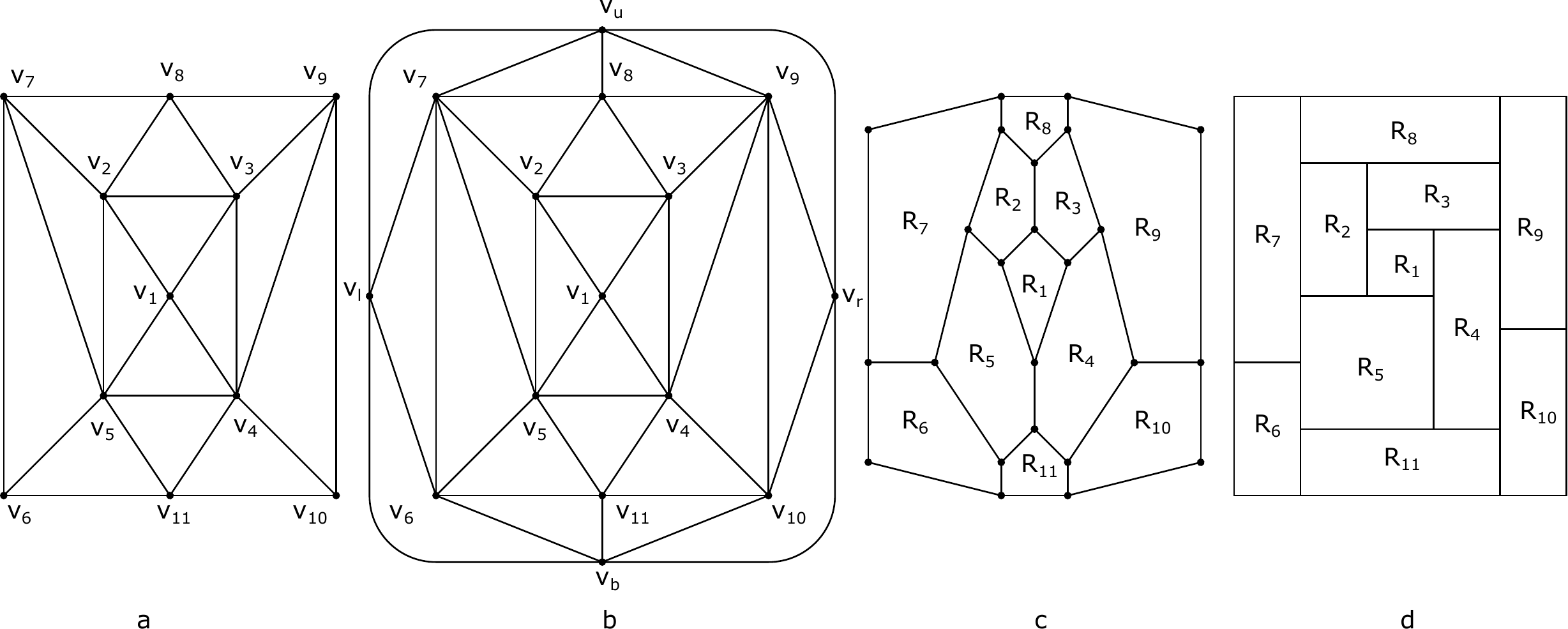}
 		\caption{Rectangular dualization: (a) plane graph, (b) extended plane graph, (c)  rectangular dual graph and (d) rectangular floorplan }
 	\label{fig:f1}
 \end{figure}
 
\subsection{\textbf{RDG application}} 
 A VLSI system structure is described by a graph where vertices correspond to component modules and edges correspond to required interconnections.  For a given  graph structure of a VLSI circuit, floorplanning is concerned with allocating space to component modules and their interconnections \cite{Heller82}. An embedding method given by Heller \cite{Heller82} enforces interconnection by abutment. Modules are designed in such a way that their connectors exactly match  with  their neighbors. Adapting this methodology, interconnections are coped with a {\it clever} design.
\par
 Due to the advancement of VLSI technology, it is extremely large. On the other hand, an RDG can handle atmost $3n-7$ interconnections  \cite{Shekhawat18}, where $n$ is the number of modules.  Consequent to this, its graph  may not necessarily be planar  and hence in modern VLSI system, component modules and interconnections can not be treated as independent entities. In such a situation, not all interconnections can be enforced by abutment. Linking the remaining interconnections with nonadjacent modules utilize additional routing space. For practicality of solution, a graph described by a VLSI circuit can be embedded  in such a way that most of the interconnections can be made  by abutment and the remaining interconnections linking with nonadjacent modules use additional routing space. For example, in Fig. \ref{fig:f2}d,  $R_7$ and  $R_9$,  $R_2$ and  $R_6$ are interconnected through  shaded areas $R_{10}$ and $R_{11}$ respectively. These routing areas are anticipated by introducing crossover vertices\footnote{These vertices are introduced at the intersection of edges if it exists in order to embed a graph as a plane graph.} at a common point of intersection of edges.
   
 The use of an RDG in floorplanning of a VLSI system  can be illustrated by the following example. Consider a graph described by  a VLSI system   as shown in  Fig. \ref{fig:f2}a. Note that input-output connections between  VLSI system and  outside world is represented by arrow heads. Although this graph  is not planar, it is planarized by adding cross over vertices as shown in Fig. \ref{fig:f2}b.  In order to satisfy the necessary adjacency requirements, new edges (red edges) have been added in Fig. \ref{fig:f2}c. After these  modifications, it is possible to construct an  RFP as shown in Fig.  \ref{fig:f2}d where a component rectangle $R_i$  is dualized  to a vertex $v_i$.

\begin{figure}[H]
	\centering
	\includegraphics[width=0.9\linewidth]{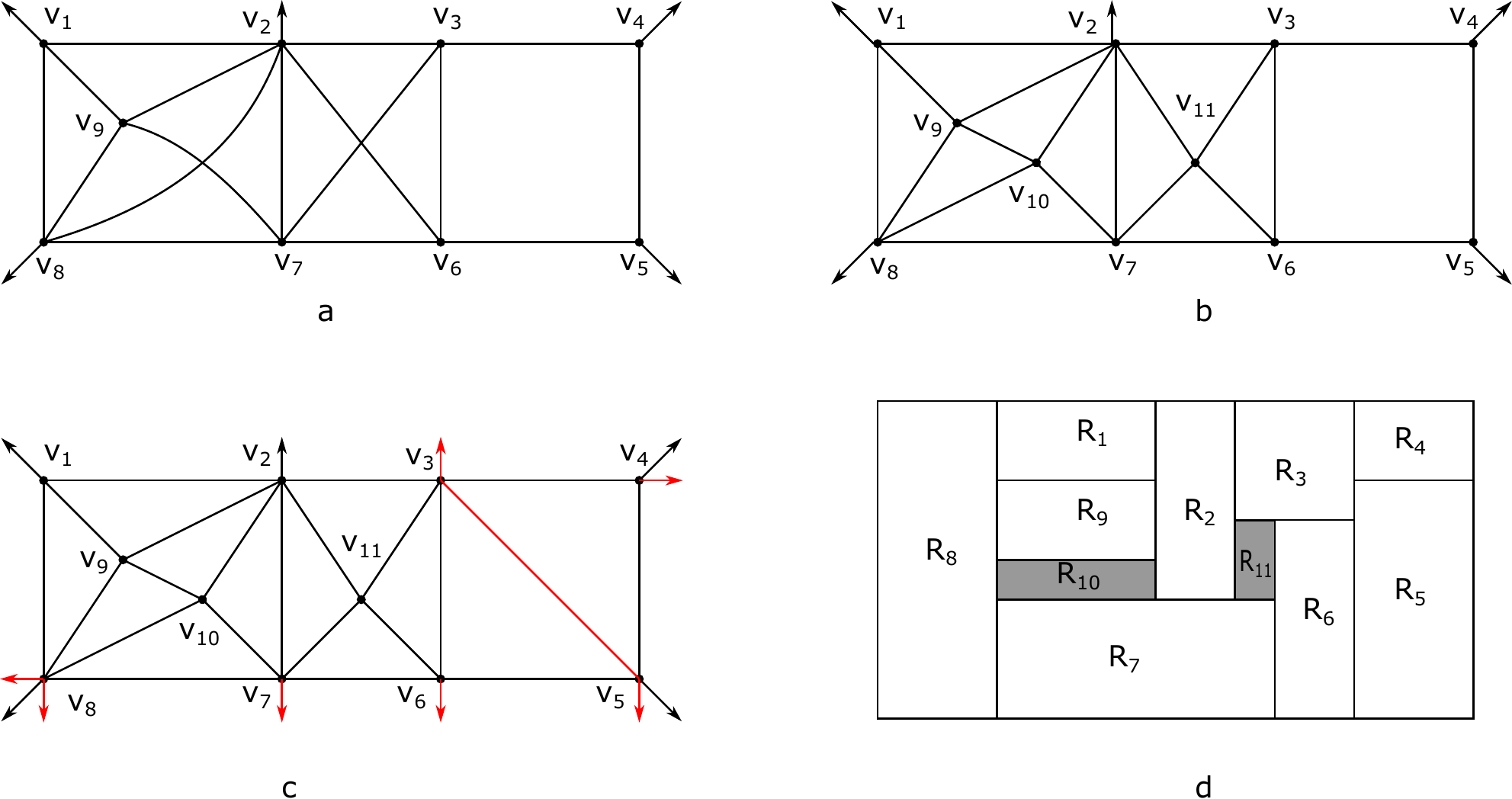}
	\caption{\rm Constructing an RFP corresponding to a plane graph described by a VLSI system.}
	\label{fig:f2}
\end{figure}

\subsection{\textbf{Previous Study}}
Investigations in the literature shows that the rectangular dualization theory \cite{Kozminski85,Bhasker87,Rinsma87,Lai90} of planar graphs is not much emphasized. It is known that every plane graph can be dualized, but not rectangularly dualized \cite{Kozminski85,Bhasker87,Rinsma87,Lai90}.  Kozminski and Kinnen \cite{Kozminski85} derived a necessary and sufficient condition for a plane triangulated graph  to be an RDG and  implemented it in quadratic time \cite{kozminski1984algorithm}.  Later, Bhasker and Sahni \cite{Bhasker87} improved this complexity to linear time implementing the rectangular dualization theory given by  Kozminski and Kinnen \cite{Kozminski85}. Rinsma \cite{Rinsma87} showed through a counter example that it is not always  possible  for a vertex-weighted outer planar graph having 4 vertices of degree 2 to be an RDG. Besides this property, there are infinite outer planar graphs that are not rectangularly dualized. In fact,  an outer planar graph having more than four critical shortcuts can not be rectangularly  dualized. This can be contradicted by our proposed Theorem \ref{thm52} in Section \ref{sec5}.  Since the graph structure of a VLSI system is never outer planar due to its large size, this theory can not be preferable for VLSI circuit's design.  The theory of rectangularly dualizable outer planar graphs plays a limited role in  building architecture also.  Lai and Leinward  \cite{Lai90} showed that solving an RFP problem of a planar graph is equivalent to a matching problem of a bipartite graph derived from the given graph. This theory relies on the assigned regions to vertices of a graph. But  this theory is not easy to  implement, i.e.,  how can we check the assignments of regions to vertices in an arbitrary given plane graph? In fact,  this theory is not  implementable until  a method for checking assignments of regions to vertices in an EPTG (extended planar triangulated graph) is known.   
 \par 
  For practical use to VLSI field,  many constructive algorithms  \cite{Bhasker88,Lai88,Kozminski88,Tang90,He93,Yeap95,Kant97,Dasgupta98,Dasgupta01,Eppstein12} based on the graph dualization theory were developed.   
 \par  
 Counting of RFPs has been remained a great issue in combinatorics \cite{Nakano01,Shen03,Yao03,Ackerman06,Reading12,Dawei14,Yamanaka17} because it produces a large solution space. To find a good condition solution in such large solution space is very hard and time consuming.  Also, in these  approaches,  attention is given to blocks-packing in the minimal rectangular area. The other major concerns such as interconnection wire length  is lagged behind.    
 \par
 With the renewed interest in floorplanning, floorplans are constructed  using rectilinear  modules (concave module)     \cite{Yeap93,He99,Chiang05,Zhang11,Alam13} also. Since  a concave rectilinear module is made up  of more than one rectangle, its design complexity   is higher than a  rectangular module (convex). Constructing a floorplan using concave rectilinear modules may decrease the  quality of the floorplan.

\subsection{\textbf{Gap in the existing work}} Motivated by the following points,  we here find a necessary and sufficient condition for a given plane graph to be an RDG. 
\begin{enumerate}
	\item[i.] Kozminski and Kinnen \cite{Kozminski85} found the following necessary and sufficient  for a separable connected graph to be an RDG.  
	\begin{theorem} \label{thm11}
	{\rm \cite[Theorem 3]{Kozminski85}	Suppose  that $\mathcal{G}$ is a separable connected plane graph with each of its  interior faces triangular.  $\mathcal{G}$ is an RDG if and only if
	\begin{enumerate}
		\item $\mathcal{G}$ has no separating triangle,
		\item block neighborhood graph (BNG) is a path,
		\item each maximal blocks\footnote{A maximal block of a graph $\mathcal{G}$ is a  biconnected subgraph of $\mathcal{G}$ which is not contained in any other block.} corresponding to the endpoints of the BNG contains at most 2 critical corner implying paths, 
		\item no other maximal block contains a critical corner implying path.	
		\end{enumerate}}
	\end{theorem}
	The proof of this theorem is not rigor, but an outline. Furthermore, we  present  a counter example showing that it does  not work properly for all separable connected graphs (see Fig. \ref{fig:f3}).
	 \begin{figure}[H]
		\centering
		\includegraphics[width=0.5\linewidth]{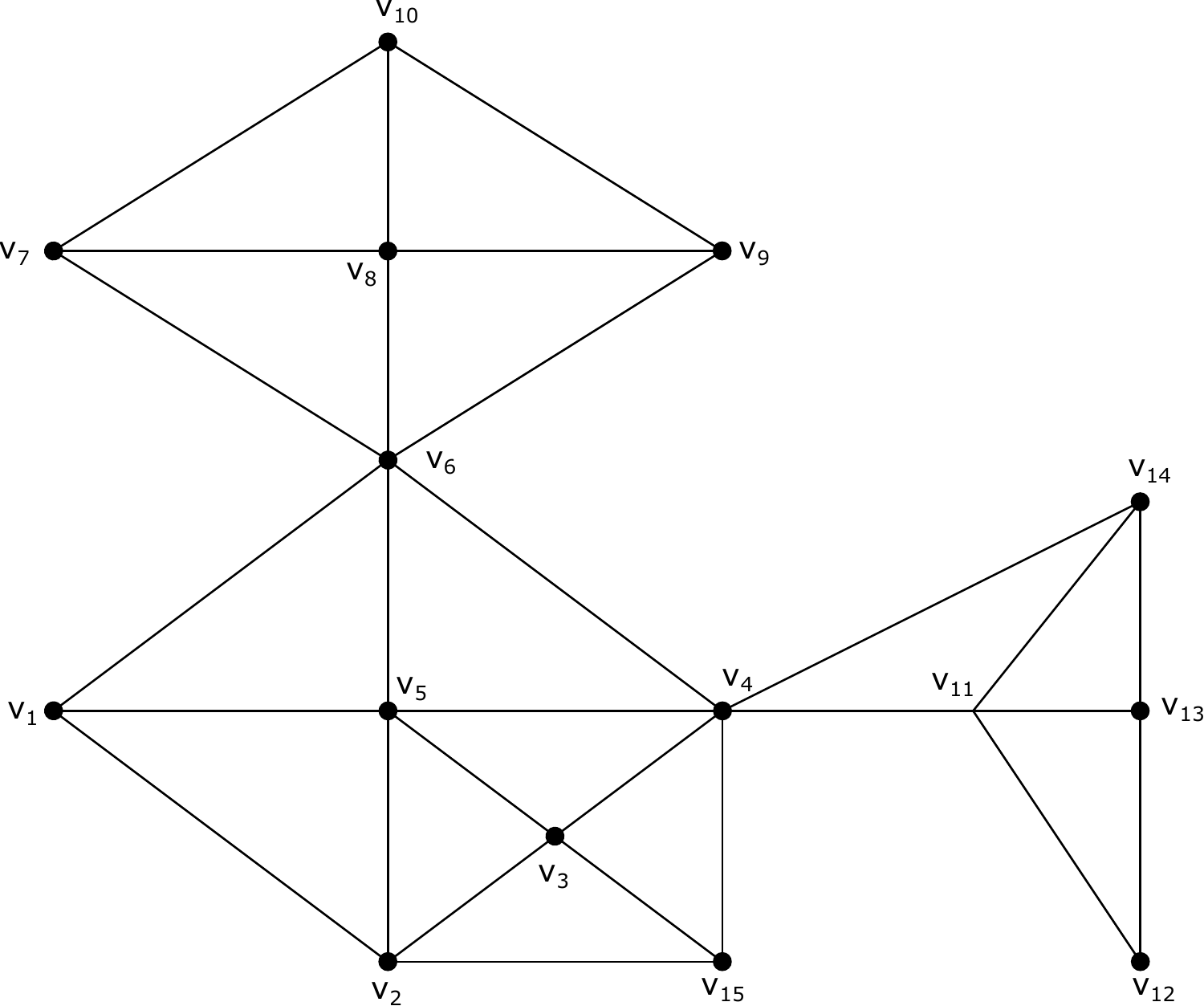}
		\caption{ \rm A counter example contradicting Theorem \ref{thm11}.}
		\label{fig:f3}
	\end{figure}
	
   Although the given graph $\mathcal{G}$ in Fig. \ref{fig:f3} satisfies all the conditions given in Theorem \ref{thm11},  it is not an RDG. Using the existing algorithm \cite{Bhasker88}, one can find an RFP for each of its  blocks. Then, an RFP for $\mathcal{G}$ can be obtained by gluing them  in a rectangular area which is not possible because of the adjacency relation of cut vertices $v_4$ and $v_6$. Corresponding to a cut vertex, there always associate a through rectangle\footnote{A through rectangle shares two sides to the exterior, but they are opposite sides.}\cite{rinsma1988existence} in an RFP for $\mathcal{G}$. But in Fig. \ref{fig:f3}, the cut-vertices  are adjacent. Hence, it is not possible to maintain rectangular enclosure while keeping $R_4$ and $R_6$ as through rectangles.
 \item[ii.] Except Theorem \ref{thm11}, there does not exist  a theorem  to check the existence of an RFP for separable connected planar graphs in the literature.
  
\item[iii.] Lai and Leinward \cite{Lai90}  derived  the following necessary  and sufficient condition for an extended plane triangulated graph (EPTG) to be an RDG:
\begin{theorem} 
{\rm  \cite[Theorem 3]{Lai90} An  EPTG is an RDG if and only if each of its triangular regions can be assigned to one of its corner vertices such that each vertex $v_i$ has exactly $d(v_i)-4$ triangular region assigned to it.}
\end{theorem}	
This theory is not  implementable until  a method for checking assignments of regions to vertices in an EPTG is known. 
\item[iv.] As discussed above Rinsma's work \cite{Rinsma87} does not fully cover the class of all rectangularly dualizable outer planar graphs. 	
	
\end{enumerate}

\subsection{\textbf{ Results }} In this paper, we find  a necessary and sufficient condition for a given plane graph to be an RDG. We show that the unbounded region of an RFP can be   realized by an unbounded rectangle  in the extended Euclidean plane. Equivalently, we can say that  an RFP can be seen as a quadrangulation of the Euclidean plane for which  we find a stereographic projection of   an RDG.

A brief description of our contribution is as follows:
In  Section \ref{sec2},  we discuss existing facts about RDGs. Section \ref{sec3} describes the extended RDG construction process. In  Section \ref{sec4}, we find stenographic projection of the dual of   an RDG   in order to extract some result pertaining to the exterior (unbounded) region of the dual.  In Section \ref{sec5}, we derive a  necessary and sufficient condition for an EPTG to be an RDG. Finally, we conclude our contribution and discuss future scope  in Section \ref{sec6}.
\par
 A list of notations used in this paper can be seen in Table 1.
\begin{table}[H]
	\centering
	\begin{tabular}{|p{2cm}|p{9.6cm}|}
		\hline
	Symbol & Description \\
		\hline
		RFP & rectangular floorplan \\
		\hline
		RDG & rectangularly dualizable graph \\
		\hline
		
		PTG & plane triangulated  graph \\
		\hline
	    EPTG & extended  plane triangulated graph \\
		\hline
		$\mathcal{G}$ & a simple  connected planar triangulated graph \\
		\hline
			$\mathcal{G}^*$ & Extended plane triangulated graph \\
		\hline
		
		$v_i$ &$i^{\text{th}}$ vertex of a graph  \\
		\hline
		$d(v_i)$ & degree of $v_i$\\
		\hline
		$(v_i,v_j)$ & an edge incident to vertices $v_i$ and $v_j$\\
		
		\hline	
		$R_i$ & $i^{\text{th}}$  rectangle (region) of an RFP (RDG) corresponding to  $v_i$ \\
		\hline
		
	\end{tabular}\\
\caption{\rm List of Notations}
\end{table}

\section{Preliminaries} 
\label{sec2}
 In this section, we  survey several facts about RFPs that would be helpful to prove our results.
 \par 
   A graph is called planar if it   can be  drawn in the Euclidean plane without crossing its edges except endpoints.  A plane graph is a planar graph with a fixed planar drawing. It splits the Euclidean plane into connected regions called faces; the unbounded region is the exterior face (the outermost face) and all other faces are interior faces. The vertices lying on the exterior face are exterior vertices and all other vertices are interior vertices. A graph is said to be  $k-$connected if it has at least $k$ vertices and the removal of fever than $k$ vertices does not disconnect the graph. If a connected graph has a cut vertex, then it is called a separable  graph, otherwise it is called a nonseparable  graph. Since floorplans are concerned with connectivity, we only consider nonseparable (biconnected) and separable connected graphs in this paper. A plane graph is called plane triangulated  graph (PTG) if it has triangular faces. A TPG may or may not have exterior face triangular. In this paper, an PTG represents a plane graph  with interior triangular faces. 
   \begin{definition} \label{def21}
   {\rm A  graph is said to be {\it rectangular graph } if each of its edges can be oriented horizontally or vertically such that it encloses a rectangular area. If the dual graph of  a planar graph is a rectangular graph, then the graph  is said to be a {\it rectangularly dualizable graphs} (RDG). In other words, a planar graph is rectangular dualizable (RDG) if its dual can be realized as a rectangular floorplan (RFP). An RFP is a partition  of a  rectangle $\mathcal{R}$ into $n$ rectangles $R_1, R_2,\dots, R_n$ provided that no four of them meet at a point.}
   \end{definition}

\begin{definition} \label{def23}
{ \rm \cite{Kozminski85} The block neighborhood graph (BNG) of a plane graph $\mathcal{G}$ is a graph where vertices are represented by biconnected components of $\mathcal{G}$ such that  there is an edge between two vertices if and only if the two biconnected components they represent, have a vertex in common.}
\end{definition} 

\begin{definition} \label{def24}
{\rm \cite{Kozminski85} A shortcut in a plane block $\mathcal{G}$ is an edge that is incident to two vertices on the outermost cycle $C$  of $\mathcal{G}$ and is not a part of $C$. A corner implying path (CIP) in   $\mathcal{G}$ is a  $v_1-v_k$ path on the outermost cycle of $\mathcal{G}$ such that it does not contain any vertices of a shortcut other than $v_1$ and $v_k$ and the shortcut $(v_1,v_k)$ is called a critical shortcut.}

\end{definition} 
 For a better understanding to Definition \ref{def24},  consider a  graph shown in Fig. \ref{fig:f4}. Edges $(v_1,v_3)$, $(v_4,v_9)$ and $(v_6,v_8)$ are  shortcuts. Paths $v_1v_2v_3v_4$ and $v_6v_7v_8$  are CIPs while path $v_9v_1v_2v_3v_4$  is  is not a CIP  since  it contains the endpoints of other shortcut $(v_1,v_3)$ and hence $(v_9,v_4)$ is not a critical shortcut. Both shortcuts $(v_1,v_3)$ and $(v_6,v_8)$ are of length 2.

\begin{figure}[H]
	\centering
	\includegraphics[width=0.7\linewidth]{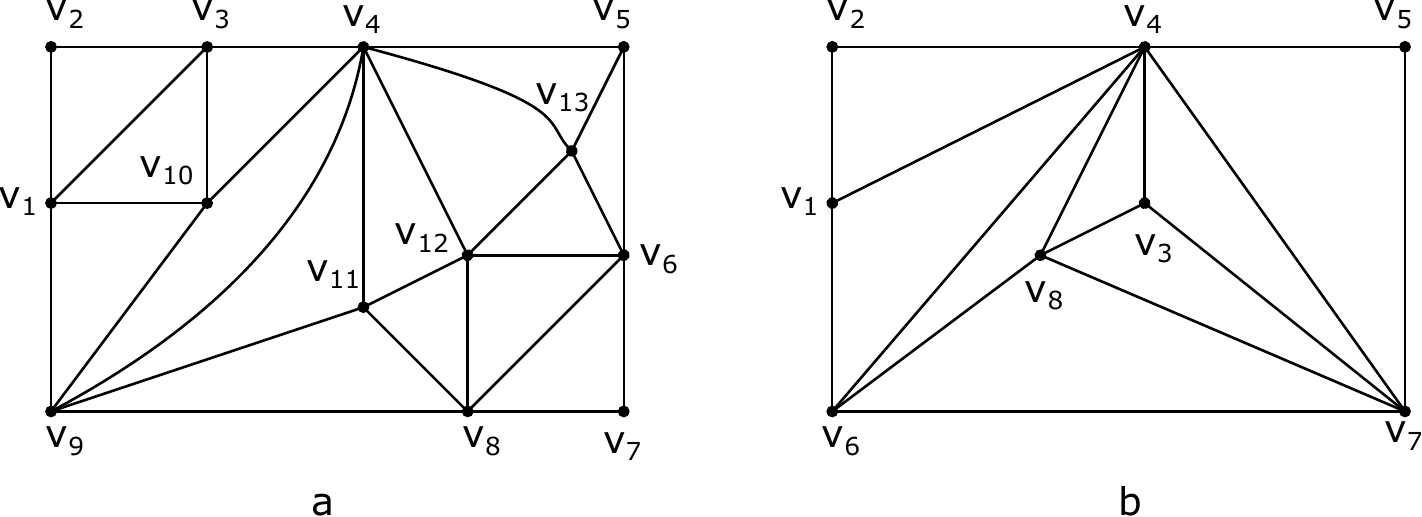}
	\caption{\rm (a) Presence of  CIPs $v_1v_2v_3$ and (b) a separating triangle $v_4v_6v_7v_4$.}
	\label{fig:f4}
\end{figure}

	{\rm  A component rectangle in an RFP is called a  corner rectangle \cite{rinsma1988existence} if its two adjacent sides are adjacent to the exterior  while  a through rectangle shares its two opposite sides to the exterior. A component rectangle in an RFP is called an  end rectangle  if its three sides are adjacent to the exterior.}	
\begin{definition} \label{def26}
{\rm A {\it separating cycle} is a cycle  in a plane graph $\mathcal{G}$ that encloses vertices inside as well as outside. If a separating cycle is of  length 3, it is called a {\it separating triangle} or a complex triangle.  We say a   separating triangle in a plane graph  is {\it critical separating triangle} if it does not  contain any other separating triangle in its interior.}
\end{definition} 
For instance,  in Fig. \ref{fig:f4}b, the cycle $v_4v_6v_7v_4$ is a separating triangle while  the cycle $v_4v_8v_7v_4$ is a critical separating triangle.

\begin{theorem} \label{thm22}
	{\rm \cite[Theorem 3]{Kozminski85} A nonseperable plane graph $\mathcal{G}$ with triangular interior faces except exterior one is an RDG if and only if  it has atmost 4 CIPs and has no separating triangle.}
\end{theorem}

\begin{theorem} \label{thm23}
	{\rm \cite{west1996introduction} A  graph $\mathcal{G}$ is 4-connected if and only if  there exist atleast 4 vertex-disjoint path between any two vertices of $\mathcal{G}$.}
\end{theorem}

\section{Extended RDG Construction} \label{sec3}

 In a graph  described by a VLSI system, vertices and edges correspond to component modules and required interconnections respectively. Communication with units outside the given system are modeled by edges having one end incident to a vertex  at the infinity (denoted by $v_{\infty}$, see Fig. \ref{fig:f5}). 
 The vertex $v_\infty$ of an RDG corresponds to  the unbounded region of its rectangular dual (RFP). 
\par 
Only planar graphs can be dualized. Whenever the graph structure is not planar, it can be made planar by adding crossover vertices until the resultant graph is planar. Such vertices are inserted at crossings of edges  in order to split the edges of a nonplanar graph. In general,  maximum interconnections by abutment and minimum through channels is used as an objective function.  Without loss of generality, we consider simple  connected planar graphs in this paper.
\par
 In a floorplan, meeting $k$-component rectangles at  a point is  called $k$-joints. Since an RFP has three joints or four joints only, its dual  has triangular or quadrangle regions only.  Abiding by common design practice,  we consider RFPs with three joints only. 
 In fact, a quadrangle region can be partitioned into two triangular regions. In such cases, some extra adjacency requests allow unrelated components in the RDG to connect, but these connections are not used  for interconnection.   
\par
Furthermore, a rectangular graph needs to be fitted in a rectangular enclosure while connecting to the outside world. Vertices that correspond to  regions next to the enclosure are called enclosure vertices \cite{Lai88} and those vertices correspond to corner regions are called corner enclosure vertices. In figure \ref{fig:f5}, vertices $v_7,v_6,v_5,v_4,v_3,v_2,v_1$ are enclosure  vertices and  $v_7,v_5,v_3,v_1$ are corner enclosure vertices. Since the enclosure has 4 sides, out of these enclosure vertices,   the  enclosure corner vertices correspond to corner rectangles or end rectangles of an RFP where a corner rectangle shares its two sides to the unbounded (exterior) region and end rectangle shares three sides to the exterior. Therefore, we need to consider atmost 4 extra edges between the selected enclosure corner vertices  and  $v_\infty$. These atmost 4 extra edges are known as {\it construction edges} \cite{Lai90}. A PTG where enclosure vertices are connected to $v_\infty$ together with 4 additional construction edges is called  an EPTG (extended planar triangulated graph). An EPTG is depicted  in Fig. \ref{fig:f5} by red edges.

It is interesting to note that the regions including unbounded region are triangulated in EPTG so that every region including unbounded region of the dual of RDG is quadrangle. This permits the enclosure to be rectangular. A detailed description of unbounded quadrangle region of the dual  can be seen in Section \ref{sec4}.  Since there is one to one corresponding between the edges of  a plane graph and its dual,  an enclosure corner vertex has parallel edges to $v_\infty$. 
In this paper, we consider a simple connected plane triangulated graph, i.e., there is no loops or parallel edges. However, some minor changes ( parallel edges between enclosure corner vertices and $v_\infty$ only) in the EPTG is done in order to choose  four construction edges. 

\begin{figure}[H]
	\centering
	\includegraphics[width=0.5\linewidth]{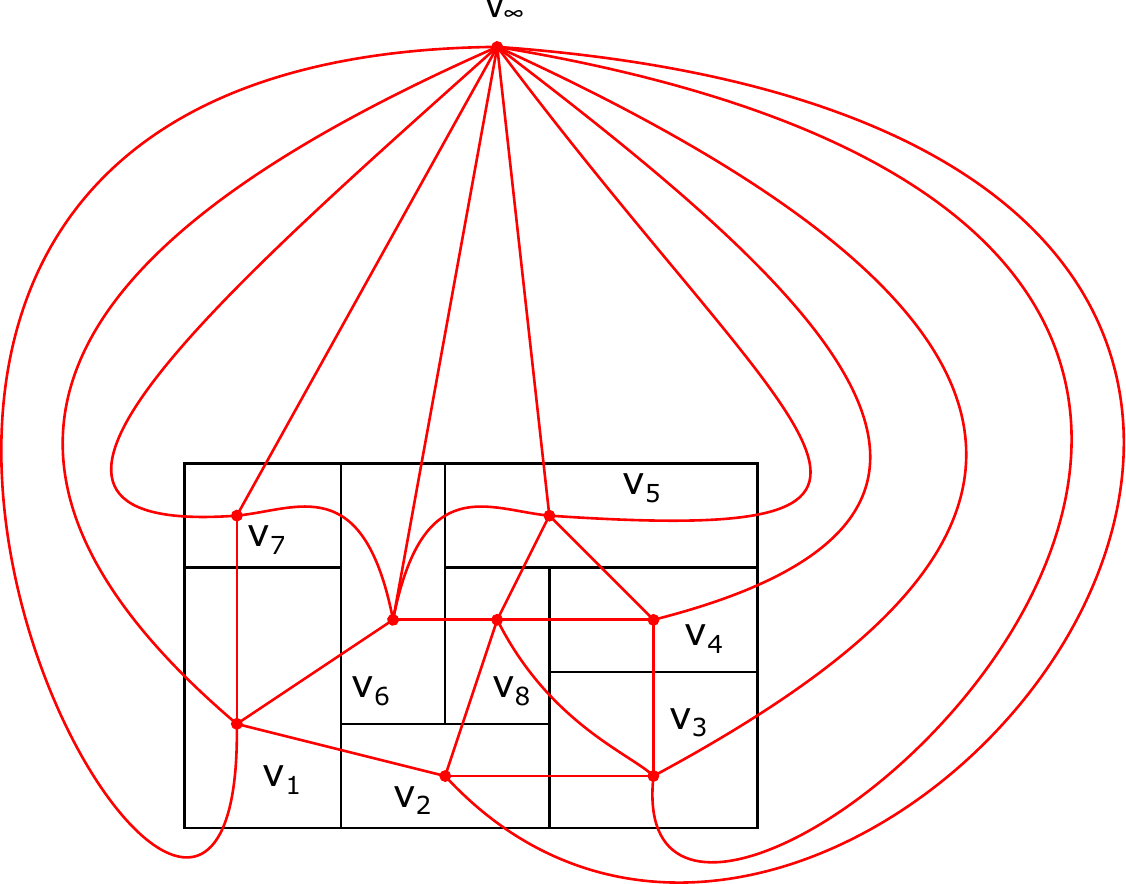}
	\caption{\rm  Construction of an extended RDG (red edges) and   its corresponding  RDG   (dark edges).}
	\label{fig:f5}
\end{figure}

\section{RDG Stereographic Projection}
\label{sec4}
In this section,  we describe  stereographic projection of a rectangular graph.
\par
A small difference between a rectangular dual graph and an RFP is that  RFP is  an edge-oriented graph such that each of its edges is either horizontally or vertically aligned together with rectangular enclosure whereas the edges of a rectangular dual graph may not  be oriented either horizontally or vertically, but they can always be  oriented  horizontally or vertically. Thus, for transforming a rectangular dual graph into be an RFP, we need to know the orientations of the edges. Once it is known, a rectangular dual graph can be converted into an RFP by orienting edges aligned along horizontal or vertical axis of the Euclidean plane. Thus, an RFP can be seen  an embedding of a rectangular dual graph. Further it is interesting to note that  an RFP is always a rectangular dual graph, but converse it not true. 
\par 
Let $\mathcal{D}$ be the rectangular dual graph  of an RDG $\mathcal{G}$. Note that a connected plane graph is a single piece made up of continuous curves (called edges) joining  their ends to pairs of the specified points (called vertices) in the Euclidean plane.  Consider a  sphere $S$  centered at $(0,0,1/2)$ having radius   $1/2$, and a fixed plane embedding $\mathcal{D}^*$ of $\mathcal{D}$ in the Euclidean plane passing through $z=0$ ($xy$-plane). Let $(0,0,1)$ be the north pole $N$ and $p$ be a point of an edge of $\mathcal{D}^*$. Draw a line segment joining the points $N$ and $p$. Let $t$ be a point where it intersects the surface of $S$. Thus we see that the point $p$ is mapped to the point $t$. In this way, the image of each of its points is a curved line on the surface of $S$ and hence each  edge of  $\mathcal{D}^*$ is mapped to a curved line on the surface of $S$. This results an embedding   of $\mathcal{D}^*$ on the surface of a sphere. 

Now, it is important to identify why the edge of $\mathcal{D}^*$ is mapped to the edges on the surface of $S$? In fact,   a connected graph is carried to a connected graph by  a continuous map. Thus being the mapping continuous, the image of $\mathcal{D}^*$  is again a plane graph on the surface of $S$ with its exterior bounded. Note that the unbounded region is now mapped into a bounded region on $S$ passing through $N$. This process is known as stereographic projection and sphere is known as Riemann sphere. But $\mathcal{D}^*$ is a rectangular dual graph. Its exterior is a four sided rectangular enclosure. This results  the unbounded region of  $\mathcal{D}^*$  corresponds to a four sided bounded region of the corresponding plane graph embedded on the surface of $S$. Consequently, when we assign horizontal or vertical orientations to the edges of  $\mathcal{D}^*$  to transform into an RFP, the unbounded region of $\mathcal{D}^*$ corresponds to an unbounded rectangle (region) $R_\infty$ passing through $\infty$.  Thus we see that the exterior of an RFP is a  rectangle  $R_\infty$ passing through $\infty$. Note that $R_\infty$ is not a part of an RFP, but is a rectangle that shares its two adjacent sides to each of its enclosure corner rectangles. Recall that a rectangle is a four-sided region with 4 right  interior angles formed by its sides. Although in case of $R_\infty$, these interior angles  can be realized to be $90^{\circ}$ by looking at it from  a point at $\infty$, otherwise we realize every interior angle to be $270^{\circ}$. The role of the point at  $\infty$ is played by $N$ and hence an alternative way is to realize right angle between two sides of the four-sided region passing through $N$ in the stereographic projection of the rectangular dual graph is the angle between the intersection of their tangents to the sides of this region.   This discussion realizes us  that an RFP is  quadrangulation of the Euclidean plane.

\section{RDG Existence Theory}
\label{sec5}
In this section, we describe the theory of RDGs. 

\begin{theorem}\label{thm51}
	{\rm A necessary  and sufficient condition for an EPTG $\mathcal{G}^*$ to be an RDG is that it is 4-connected and has atmost 4 critical separating triangles passing through $v_\infty$.}	
\end{theorem}

\begin{proof}
\textbf{Necessary Condition.} Assume that  $\mathcal{G}^*$ is an RDG. Then it has a rectangular dual graph $\mathcal{D}$. Let $v_i$ be a vertex of $\mathcal{G}^*$ dual to some interior region $R_i$ of $\mathcal{D}$. Since every region of  $\mathcal{D}$ is four-sided,  atleast 4  regions are required to fully enclose an interior region of  $\mathcal{D}$. This implies that  $R_i$ is surrounded by atleast 4 regions of $\mathcal{D}$ and hence $v_i$ is adjacent to atleast 4  vertices of $\mathcal{G}^*$, i.e.,  $d(v_i) \geq 4$.  Let $v_e$ be a vertex of $\mathcal{G}^*$ dual to an enclosure (exterior) region $R_e$ of $\mathcal{D}$. There arise  two  possibilities: 
\begin{itemize}
    \item $R_e$  surrounds exactly its two sides with $R_\infty$ if it is an enclosure corner region,
    \item $R_e$ surrounds exactly its one side with $R_\infty$ if it is not an enclosure corner region.
\end{itemize}
 In the first case, $R_\infty$ surrounds the two sides of $R_e$. There are two edges between $v_\infty$ and $v_e$ where $v_\infty$ corresponds to $R_\infty$. The remaining two sides of $R_e$ are surrounded by atleast two interior regions other than  $R_\infty$.  This implies that $d(v_e) \geq 4$. In the second case, only one side of $R_e$ is surrounded by $R_\infty$ and the remaining sides are surrounded by atleast three interior regions.  This implies that $d(v_e) \geq 4$.  Since $v_e$ and $v_i$  are  arbitrary vertices of $\mathcal{G}^*$, $\mathcal{G}^*$ is 4-connected. This proves the first condition.
\par
As discussed in Section \ref{sec4}, $\mathcal{R}_\infty$ surrounds exactly its two adjacent sides to each of the  enclosure corner regions of $\mathcal{D}$ and  exactly one side to the remaining exterior regions of $\mathcal{D}$. Let $v_c$ be a vertex of $\mathcal{G}^*$ dual to an enclosure corner region of $\mathcal{D}$. We have already shown that  $\mathcal{G}^*$ is 4-connected, i.e.,  $d(v_c)\geq 4$, $\forall v_c  \in \mathcal{G}^*$. If $d(v_c)= 4$, then two adjacent sides of $R_c$ are surrounded by $R_\infty$ whereas the remaining two sides of $R_c$ are surrounded by  two regions  $R_a$ and $R_b$. Clearly, $R_a$ and $R_b$ are the enclosure regions. Since $\mathcal{G}^*$ is an EPTG, every region of  $\mathcal{G}^*$  is triangular. This implies that $R_a$ and $R_b$ are adjacent. Consequently, there is a   separating triangle passing through $v_\infty$ and vertices that are dual to $R_a$ and $R_b$. Clearly, it encloses exactly one vertex $v_c$. This implies that there is no  separating triangle inside this separating triangle and hence it is a critical separating triangle. This situation is depicted  in  Fig. \ref{fig:f6}a. If $d(v_c)> 4$, there are atleast three interior regions that surround  $R_c$. Vertices that are dual to these  interior regions together with $v_\infty$ is a cycle of  length atleast 4.  Only possibility for the existence of  a critical separating triangle passing through $v_\infty$ and enclosing  $v_c$ is depicted in Fig. \ref{fig:f6}b.  Now it is evident that there is atmost one critical separating triangle passing through  $v_\infty$ corresponding  to each enclosure  corner region. Since a rectangular graph has atmost  four enclosure corner regions, there can be atmost 4 critical separating triangles passing through $v_\infty$. This proves the second condition.
 
\begin{figure}[H]
	\centering
	\includegraphics[width=0.7\linewidth]{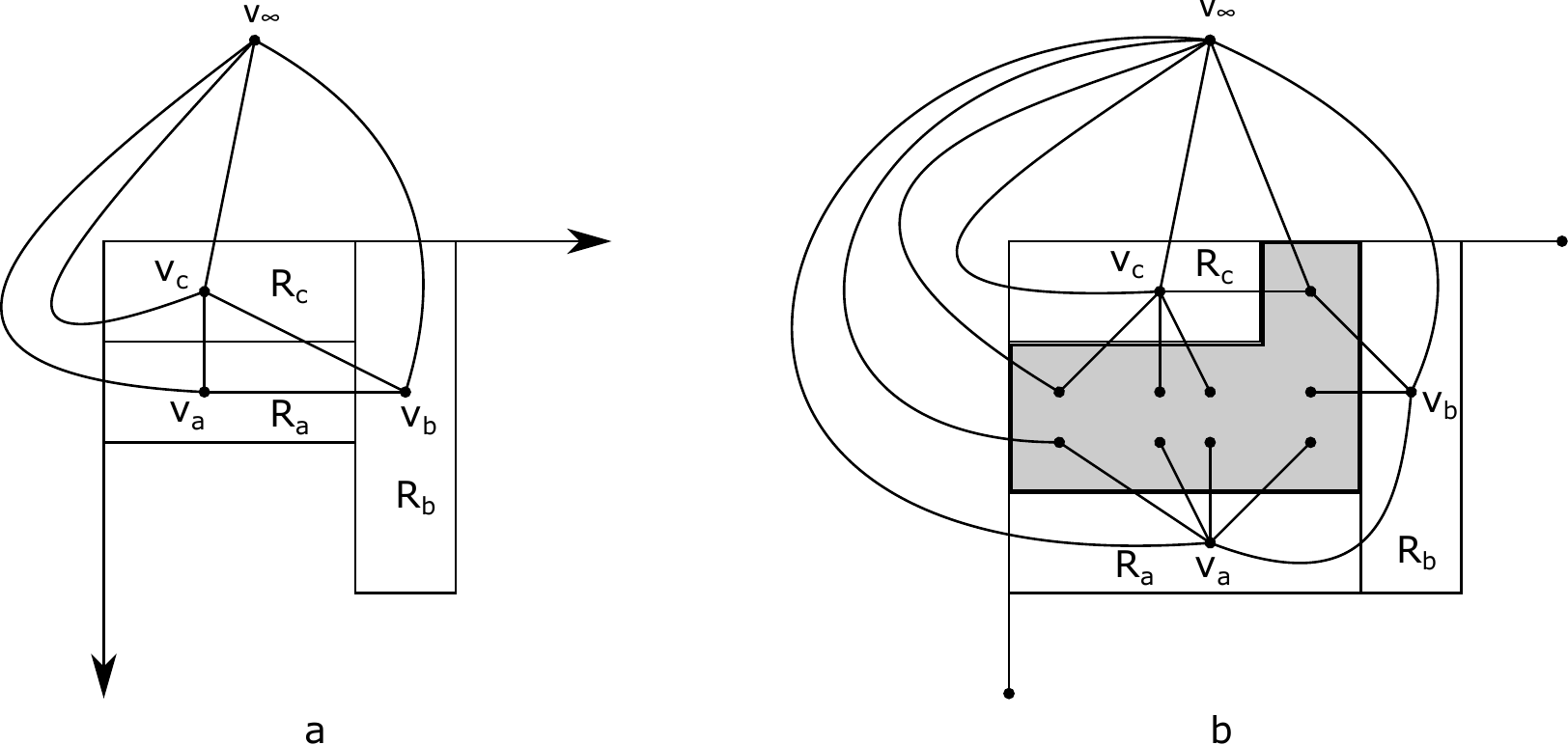}
	\caption{\rm Two possibilities of a critical separating triangle enclosing an enclosure corner vertex.}
	\label{fig:f6}
\end{figure}
\textbf{Sufficient Condition.}  Assume that the given conditions hold. 
We prove the result by applying the induction method on the vertices of $\mathcal{G}^*$. Recall that an EPTG contains atleast two vertices. Let $n$ be the number of vertices of $\mathcal{G}^*$. If $n=2$, then it is a graph consisting of a single edge and hence it is an RDG. Let us assume that $n>2$ and the result holds for $n-1$ vertices, i.e., every $(n-1)$-vertex EPTG satisfying the given conditions is an RDG. In order to complete induction, we need to  prove that $n$-vertex EPTG $\mathcal{H}$ satisfying the given conditions is an RDG. Since there can be atmost four critical separating triangles in $\mathcal{H}$, there arise two possibilities: (1) there are exactly three edges between $v_\infty$ and atleast one of the  enclosure vertices, (2) there are exactly two edges between $v_\infty$ and each  enclosure corner vertex. Let $v_i$ be an enclosure corner vertex of $\mathcal{H}$ and $A =\{ v_1, v_2,\dots, v_t\}$ be the set of vertices adjacent to $v_i$. 
\par
Consider the first case, i.e., there exist edges $(v_i,v_\infty)$, $(v_i,v_p)$, $(v_i,v_q)$ where  vertices $v_p$ and $v_q$ are incident to $v_\infty$ as shown in Fig. \ref{fig:f7}a.  Construct a new EPTG $\mathcal{H}_1$ by deleting  $v_i$ together with the incident edges and introducing  new edges $(v_\infty,v_1)$, $(v_\infty,v_2) \dots (v_\infty,v_t)$  (see Fig. \ref{fig:f7}b). We prove that $\mathcal{H}_1$ satisfies the given conditions stated in the theorem.
\par 
Consider two vertices $v_a$ and $v_b$ of $\mathcal{H}$ such that $i \neq a,b$. As $\mathcal{H}$ is 4-connected, by Menger's theorem, there exist  four vertex-disjoint paths between $v_a$ and $v_b$. Choose each path of the shortest possible length. If none of these paths uses the edges $(v_i,v_p)$ and $(v_i,v_q)$, then the same path would exist in $\mathcal{H}_1$ with the edge $(v_\infty,v_k)$, $(1 \leq k\leq t)$ substituted in the place of $(v_k,v_i)\cup (v_i,v_\infty)$ if they occur in the path. Otherwise suppose that one of the four paths passes through $(v_i,v_p)$. Being the shortest possible path, it can not pass through $v_\infty$ or $v_k$, $(1 \leq k\leq t)$. Consequently, it must use the edge  $(v_i,v_q)$. If a path  passes through $v_\infty$, it would pass through $v_p$ or $v_q$, contradicting to the facts that path  is the shortest. Thus vertex $v_\infty$ is not used by any of the four paths. Now by substituting the part  $(v_i,v_p)\cup (v_i,v_q)$  of the path in $\mathcal{H}$ by $(v_p,v_\infty)\cup (v_\infty,v_q)$ in $\mathcal{H}_1$, we can obtain 4 vertex-disjoint paths in $\mathcal{H}_1$ also. Then by Menger's theorem,  $\mathcal{H}_1$ is 4-connected. 
\par
Next we claim that  the number of critical separating triangles in $\mathcal{H}_1$ can not be more than the number of critical separating triangles in $\mathcal{H}$. As discussed in the necessary part that there is atmost one  critical separating triangle enclosing an enclosure corner vertex and  $\mathcal{H}$ has three enclosure corner vertices, there are atmost three critical separating triangles in $\mathcal{H}$.  Then the only possibility of occurring a separating triangle in $\mathcal{H}_1$ is as follows.  If  an enclosure vertex $v_l$ is incident to both   $v_p$, $v_k$ where  $v_k \in A$, then there exists a separating triangle in $\mathcal{H}_1$ passing through $v_k$, $v_p$ and $v_l$. Similarly, there can be  another separating triangle in $\mathcal{H}_1$ passing through $v_t$, $v_q$ and $v_s\in A$.  thus there can be atmost two new separating triangles in $\mathcal{H}_1$. If there exists a critical separating triangle $T_c$ containing $v_i$ in $\mathcal{H}$, then there are three possibilities: (1) there no longer remains $T_c$ in $\mathcal{H}_1$, (2) $T_c$ is contained in one of the  new created separating triangles in $\mathcal{H}_1$, and (3) One of the new created separating triangle is contained in  $T_c$. All these possibilities show that there can not be more than four critical separating triangles in $\mathcal{H}_1$.  This shows that $\mathcal{H}_1$ has  atmost 4 critical separating triangles.   Thus, $\mathcal{H}_1$ has $n-1$ vertices satisfying the given conditions.  By induction hypothesis, $\mathcal{H}_1$ is  an RDG and hence admits an RFP. This RFP can be transformed to  another RFP by adjoining a region $R_i$ (corresponding to $v_i$) as shown in Fig. \ref{fig:f7}c. Then the resultant  RFP corresponds to $\mathcal{H}$. Hence $\mathcal{H}$ is an RDG.
\par
Consider the second case. In this case, $\mathcal{H}$ appears as shown in Fig. \ref{fig:f8}a with atleast four more vertices $v_1,v_2,v_3$ and $v_4$. Consider the four enclosure corner vertices $v_1$, $v_2$, $v_3$ and $v_4$ as shown in Fig. \ref{fig:f8}a. Now we show that there is a separating cycle $C$ passing through $v_i$, $v_\infty$ and an enclosure vertex $v_d$  but not passing through $v_3$ or $v_4$ such that the removal of vertices of $C$ from  $\mathcal{H}$ disconnects it into two connected pieces, each containing atleast one vertex.
\par
If there is an edge $(v_1,v_3)$ in $\mathcal{H}$, there is a separating cycle passing through $v_1,v_3$ and $v_\infty$. In this case, $\mathcal{H}$ is separated into two parts, one of which contains atleast $v_2$ and another contains atleast $v_4$.
\begin{figure}[H]
\centering
\includegraphics[width=0.9\linewidth]{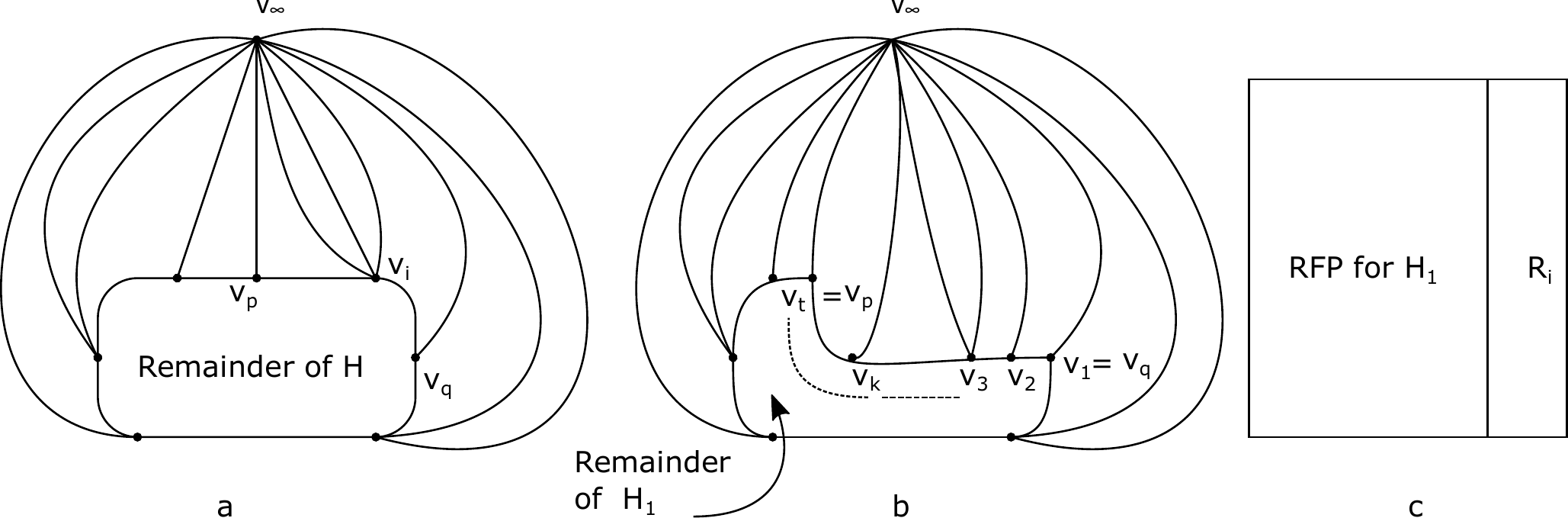}
\caption{\rm (a) Sketch of the graph $\mathcal{H}$ when there are three edges between $v_\infty$ and enclosure vertex $v_i$, (b) sketch of the graph $\mathcal{H}_1$ when  there are exactly two edges between each enclosure corner  vertex  and $v_\infty$, and (c) the construction of an rectangular dual for $\mathcal{H}$.}
\label{fig:f7}
\end{figure}
If there is no edge $(v_1,v_3)$ in $\mathcal{H}$. All  vertices adjacent to $v_3$ lie on a path $y_1y_2\dots y_k$ where $y_1$ and $y_k$ are the enclosure vertices. Let $y_kx_1x_2\dots v_2$ be a path of the enclosure vertices starting from $y_k$ and ending with $v_2$. Then $C=ty_1y_2\dots y_kx_1x_2\dots v_2$ is a separating cycle which separates $\mathcal{H}$ into two parts, one of  which atleast contains $v_1$ and another contains atleast $v_3$.
\begin{figure}[H]
	\centering 
	\includegraphics[width=0.9\linewidth]{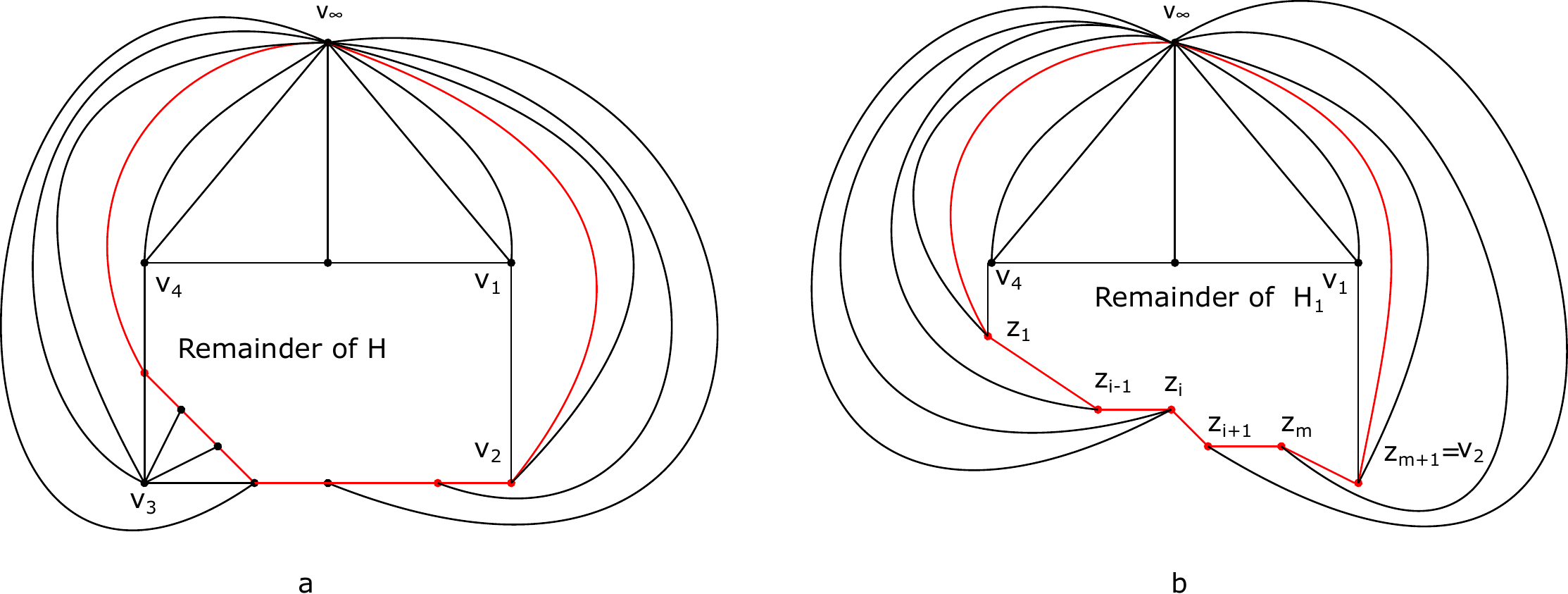}
	\caption{\rm (a) A separating cycle shown by red edges and (b) the appearance of $\mathcal{H}_u$.}
	\label{fig:f8}
\end{figure}
Once a separating cycle exists, there also exists the shortest separating cycle $C_s=v_\infty z_1z_2\dots z_mz_{m+1}$. This situation is depicted in \ref{fig:f8}b. Without loss of generality, suppose $C_s$ separates $v_1$ and $v_3$. Construct an EPTG $\mathcal{H}_u$ from the subgraph contained in the interior of $C$ by adding a vertex $v_\infty$ and edges between $v_\infty$ and enclosure  vertices of this subgraph. The new edges in  this construction are $(v'_\infty, z_1)$, $(v'_\infty, z_2)$, \dots $(v'_\infty, z_{m+1})$. Now we show that $\mathcal{H}_u$ satisfies the given conditions. Only possibility for creating a separating triangle is a triangle $z_iz_{i+1}v_\infty$ for $1 \leq i \leq n$. If there would exist an edge $(z_i,z_{i+1})$ in $\mathcal{H}_u$, then it contradicts that $C_s$ is the shortest separating cycle. Therefore, any cycle in $\mathcal{H}_u$ is of length atleast 4 and consequently, $\mathcal{H}_u$ is 4-connected and can not have more than 4 separating triangles. By induction hypothesis, $\mathcal{H}_u$ is an RDG. Similarly, we can show that the EPTG $\mathcal{H}_b$ constructed from the the remaining part of  $\mathcal{H}$ is  an RDG. Then the corresponding RFP  can be placed one above the other and can be merged after  applying homeomorphic transformation so as to preserve orthogonal directions of the edges such that the resultant floorplan is an RFP of $\mathcal{H}$ as shown in Fig. \ref{fig:f9}. This completes the induction process and hence completes the proof.
\end{proof}

\begin{figure}[H]
	\centering
	\includegraphics[width=0.7\linewidth]{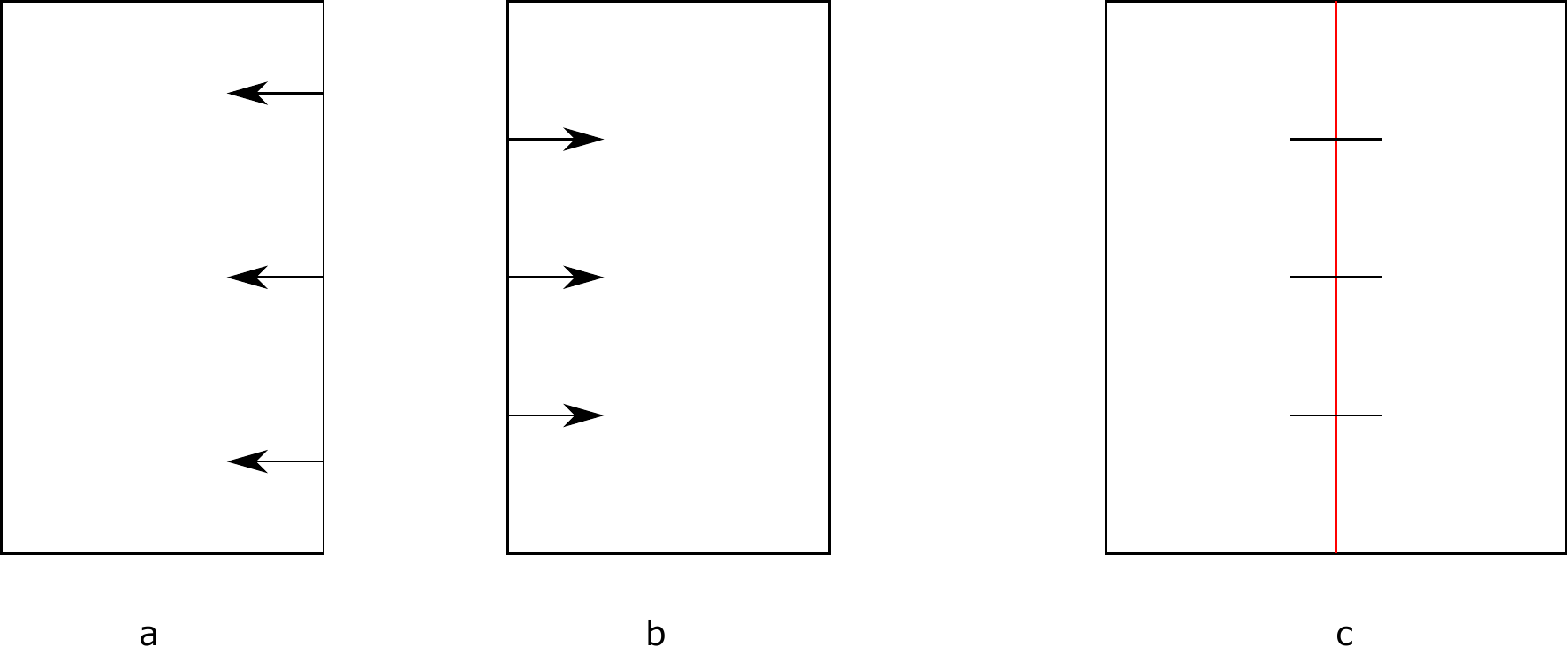}
	\caption{\rm (a) Merging two  RDGs of $\mathcal{H}_u$ and $\mathcal{H}_b$ into  an RDG for $\mathcal{H}$}
	\label{fig:f9}
\end{figure}

Now we turn our attention to derive a necessary and sufficient condition for a PTG to be an RDG. A plane graph can be either nonseparable graph (block) or a separable connected graph. A disconnected graph is also a separable graph. However, we are not considering this case since RFP are concerned with connectivity.

\begin{theorem}\label{thm52}
	{\rm A necessary  and sufficient condition for a nonseparable PTG $\mathcal{G}$ to  an RDG is that it is 4-connected and  has atmost 4 critical shortcuts.}	
\end{theorem}

\begin{proof}
\textbf{Necessary Condition.} Assume that $\mathcal{G}$ is an  RDG. Then it admits an RFP  $\mathcal{F}$.  Let $v_i$ be an interior vertex of $\mathcal{G}$ dual to a  rectangular region $R_i$ of $\mathcal{F}$. Recall that there require atleast 4 component rectangular regions to  surround a  rectangular region in an RFP. Therefore,  there exist atleast 4 rectangular regions in $\mathcal{F}$  enclosing $R_i$. Then  $v_i$ is adjacent to atleast 4 vertices, i.e.,   $d(v_i) \geq 4$. Since $v_i$ is an arbitrary interior vertex of $\mathcal{G}$, $\mathcal{G}$ is 4-connected. 
\par	
To the contrary, if there exist 5 critical shortcuts in $\mathcal{G}$, the corresponding EPTG $\mathcal{G}^*$ would contain 5 critical separating triangles, each passing through exactly one critical shortcut. This is a contradiction to  Theorem \ref{thm51}. This shows that $\mathcal{G}$ can not have   more than 4 critical shortcuts.	
\par
\textbf{Sufficient Condition.} Assume that the given conditions hold. 
Choose 4 enclosure corner vertices, each on the path joining the endpoints of the critical shortcut lying on its outermost cycle but not as the endpoints of these paths. If the number of   critical shortcuts are less than 4, choose the remaining enclosure corner vertices randomly among enclosure vertices.    Join each of  these 4 vertices to  $v_\infty$ by two parallel edges and join each of  the remaining $n-4$ enclosure vertices to $v_\infty$ by a single edge. This  constructs an EPTG $\mathcal{G}^*$ satisfying all the conditions given in Theorem \ref{thm51}. Hence $\mathcal{G}$ is an RDG. This completes the  proof.
\end{proof}

\begin{theorem}\label{thm53}
{\rm A necessary  and sufficient condition for a separable connected PTG $\mathcal{G}$ to be an RDG is that:
		
\begin{enumerate}[i.]
\item  each of its blocks  is 4-connected,
\item BNG is a path,
\item  both endpoints of an exterior edge of each of its blocks are not cut vertices,	
\item each maximal blocks corresponding to the endpoints of
the BNG contains at most 2 critical shortcuts, not passing through cut vertices, 
\item  Other remaining  maximal blocks do not contain a critical shortcut,  not passing through a cut vertex.
\end{enumerate} }	
\end{theorem}

\begin{proof}
\textbf{Necessary Condition.} Assume that $\mathcal{G}$ is an RDG. 	
The proof of the first  condition is a  direct consequence  followed by  Theorem \ref{thm51}. The BNG of $\mathcal{G}$ has the following possibilities:
  \begin{enumerate}[i.]
  	\item it can be path,
  	\item it can be a cycle of length $\geq 3$,
  	\item it can be a tree.	
  \end{enumerate}
To the contrary, suppose that the BNG is a cycle of length atleast $3$. This implies that atleast three blocks share some cut vertex $v_c$ of $\mathcal{G}$. The construction of an EPTG  $\mathcal{G}^*$ create more than 4 critical separating triangles, each  passing through $v_c$, $v_\infty$, and a vertex adjacent to $v_c$ that belongs to  the outermost cycle of each block. This situation can be depicted in Fig. \ref{fig:f10}a. Then by Theorem \ref{thm51}, $\mathcal{G}$ no longer is an RDG.  A similar argument can be applied when it is a tree. This situation can be depicted in Fig. \ref{fig:f11}a. Thus, the BNG is left with one possibility, i.e., the BNG is a path.
\par
To the contrary, suppose that both the endpoints of an exterior edge $(v_i,v_j)$ of  a block are cut vertices, then there are more than 4  critical  separating triangles  passing through $v_i$, $v_j$  and $v_\infty$  in $\mathcal{G}^*$, which is a contradiction to  Theorem \ref{thm51}. Hence both the endpoints of an exterior edge  of  a block can not be cut vertices simultaneously.
\par
Let $M_i$ be a maximal block corresponding to the endpoints of
the BNG. Since $\mathcal{G}$ is an RDG, each of its block is an RDG. Suppose that $M_i$ is an RDG. Then it admits an RFP $\mathcal{F}_i$. It can be easily noted that out of 4 corner rectangular regions of $\mathcal{F}_i$, only two can be the corner rectangular regions of $\mathcal{F}$. Then there can be atmost two critical separating triangles in $\mathcal{G}^*$ and hence there can be atmost two critical shortcuts in each $M_i$. This implies that the second condition holds. Also, any other maximal block  of the BNG can not share critical separating triangles since any corner rectangular region in $\mathcal{F}$ is an RFP. This implies that no other maximal block has  a critical separating triangle in $\mathcal{G}^*$ and hence there is no  critical shortcut in the remaining maximal blocks.  
\begin{figure}[H]
	\centering
	\includegraphics[width=0.7\linewidth]{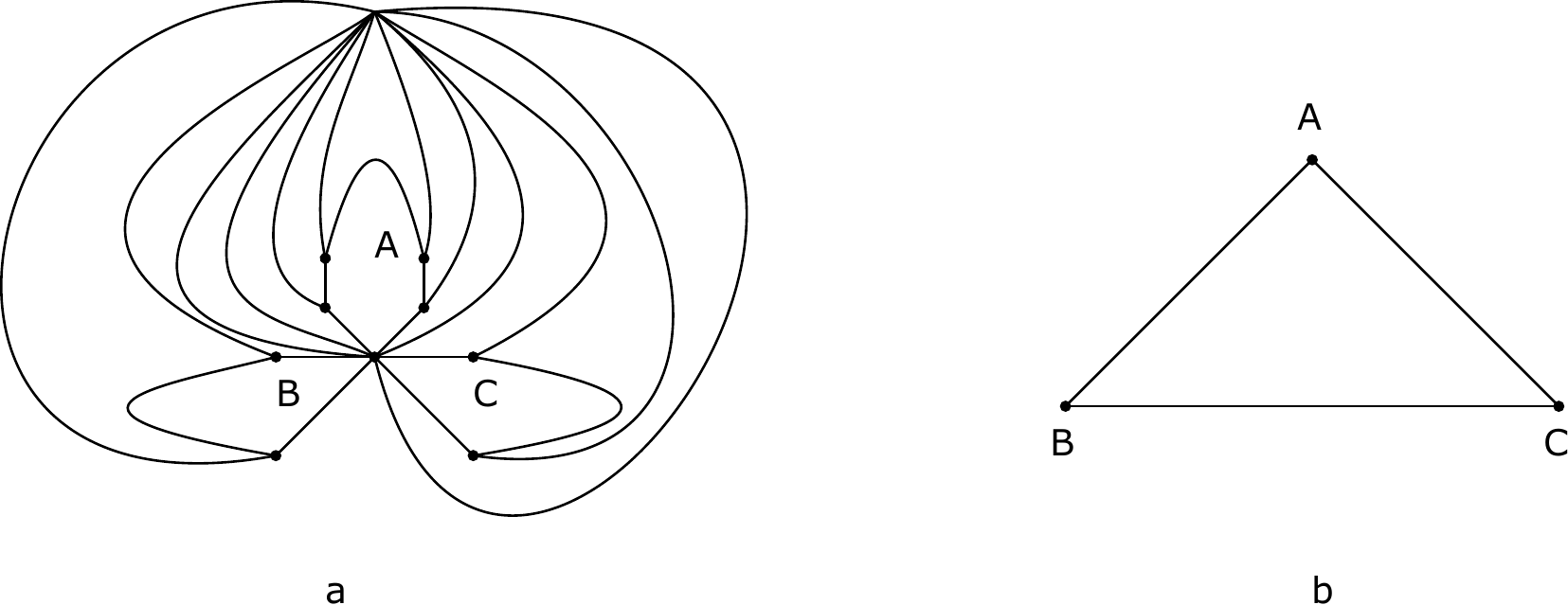}
	\caption{\rm (a) A separable connected graph constituted by three blocks A, B and C, and (b) its BNG. Here only the outermost cycles of the blocks  are shown. }
	\label{fig:f10}
\end{figure} 
\textbf{Sufficient Condition.} Assume that the given conditions hold. The first condition shows that $\mathcal{G}^*$ is 4-connected.
The remaining conditions show that there are atmost four critical separating triangles in $\mathcal{G}^*$. By Theorem \ref{thm51}, $\mathcal{G}$ is an RDG. Hence the proof.
\end{proof}

\begin{figure}[H]
	\centering
	\includegraphics[width=0.7\linewidth]{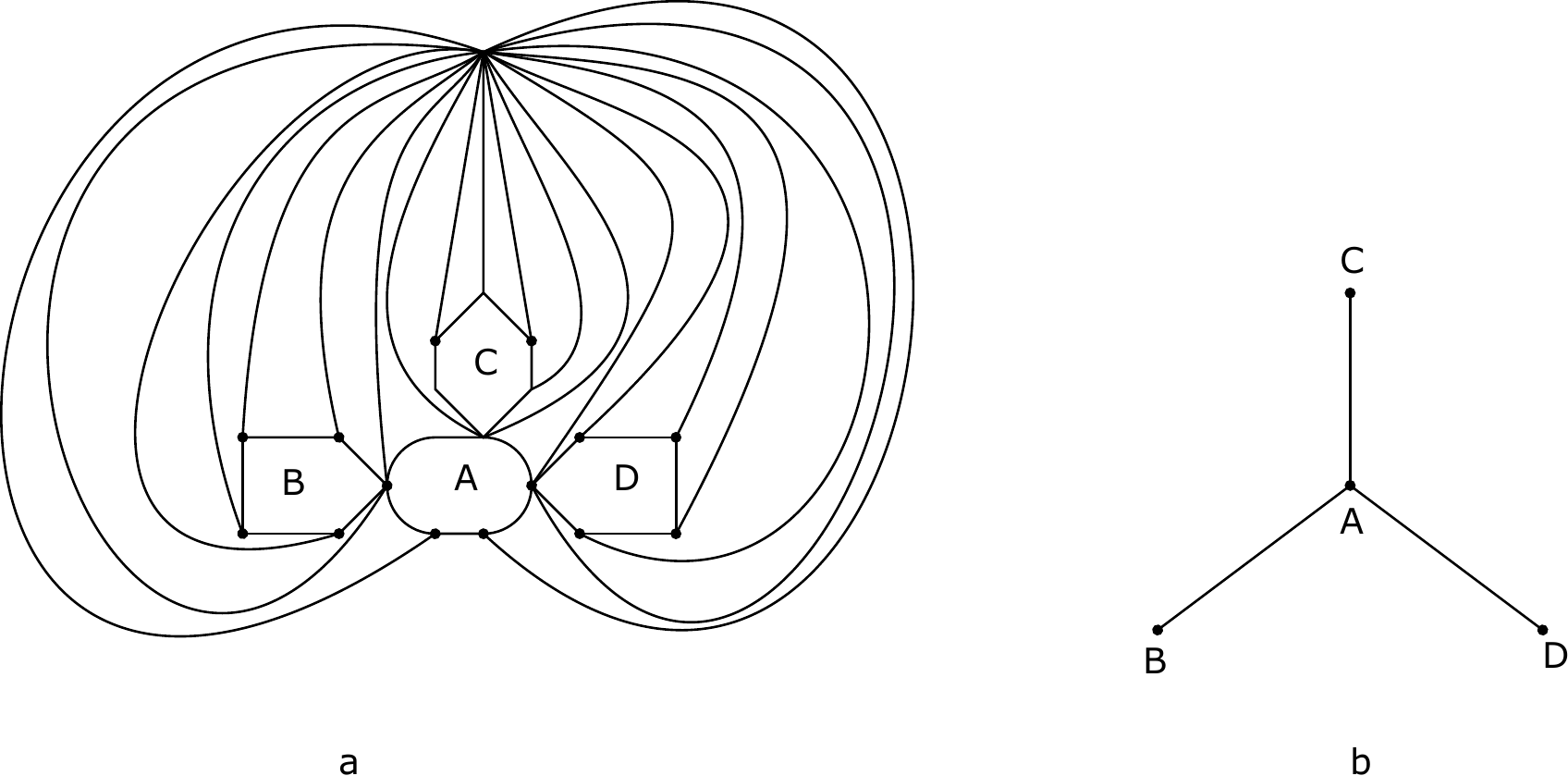}
	\caption{\rm  A separable connected graph constituted by three blocks A, B, C and D, and (b) its BNG. Here only the outermost cycles of the blocks  are shown.}
	\label{fig:f11}
\end{figure}

\section{Concluding remarks and future task} \label{sec6}
 We developed graph theoretic characterization of RFPs.  We reported that the existing RDG theory may fail in some cases. 
 Hence, we proposed a new RDG theory which  is  easily implementable and it simplifies the floorplan construction process of the VLSI circuits as well architectural buildings. 
\par
In future, it would be interesting to transform a nonRDG into  an RDG by removing those edges  which violates the RDG property and then adding new edges (maintaining RDG property) in such a way that the distances of endpoints of  the deleted edges can be minimized. This idea  would be useful  in reducing the interconnection wire-lengths as well as in complex buildings,  it gives  the shortest possible paths for those pairs of rooms which is impossible to directly connect.

 \bibliographystyle{elsarticle-num} 
\bibliography{references}

\end{document}